\documentclass[preprint,12pt]{elsarticle}

\usepackage{amssymb}
\usepackage{amsbsy}
\usepackage{amsmath}
\usepackage{color}

\journal{\textsf{arXiv}, typeset with elsarticle.cls}

    \newcommand{\floor}[1]{\lfloor#1\rfloor}
    
    \newcommand{\norm}[1]{\| #1 \|}

    \newcommand{\RR}{\mathbb{R}}
    
    \newcommand{\EE}{\mathbb{E}}
    \newcommand{\ZZ}{\mathbb{Z}}
    \newcommand{\NN}{\mathbb{N}}

    \newcommand{\Exp}{\operatorname{E}}
    \newcommand{\E}{\Exp}

    \renewcommand{\Pr}{\operatorname{P}}

    \newcommand{\dto}{\xrightarrow{d}}
    \newcommand{\wto}{\xrightarrow{w}}
    \newcommand{\vto}{\xrightarrow{v}}
    \newcommand{\fidi}{\xrightarrow{\text{fidi}}}

    \newcommand{\eind}{\stackrel{d}{=}}
    \newcommand{\rmd}{\mathrm{d}}

    \newcommand{\be}{\begin{equation}}
    \newcommand{\ee}{\end{equation}}
    \newcommand{\som}{{\textstyle\sum}}
    \renewcommand{\le}{\leqslant}
    \renewcommand{\ge}{\geqslant}
    \renewcommand{\leq}{\le}
    \renewcommand{\geq}{\ge}

  \newtheorem{theorem}{Theorem}[section]
  \newtheorem{lemma}[theorem]{Lemma}
  \newdefinition{cond}[theorem]{Condition}
  \newdefinition{remark}[theorem]{Remark}
  \newdefinition{example}[theorem]{Example}
  \newproof{proof}{Proof}

\numberwithin{equation}{section}

\begin{document}

\begin{frontmatter}

\title{A Multivariate Functional Limit Theorem in Weak $M_{1}$ Topology}

\author[fn1]{Bojan Basrak}
\ead{bbasrak@math.hr}

\author[fn2]{Danijel Krizmani\'{c}}
\ead{dkrizmanic@math.uniri.hr}

\address[fn1]{University of Zagreb, Department of Mathematics, Bijeni\v{c}ka 30, 10000 Zagreb, Croatia}
\address[fn2]{University of Rijeka, Department of Mathematics, Radmile Matej\v{c}i\'{c} 2, 51000 Rijeka, Croatia}

\begin{abstract}

We show a new functional limit theorem for weakly dependent regularly varying sequences of random
vectors. As it turns out, the convergence takes place in the space of
 $\mathbb{R}^{d}$ valued c\`{a}dl\`{a}g functions
  endowed with the so-called weak $M_{1}$ topology.
 The  theory is illustrated on two examples. In particular, we demonstrate
why such an extension of Skorohod's $M_1$ topology
is actually necessary for the limit theorem to hold.
\end{abstract}


\end{frontmatter}

\section{Introduction}
\label{intro}

Literature in theoretical probability and statistics abounds with studies of the limiting
behaviour of partials sums, mostly in the case of stationary sequences with so-called light tails.
On the other hand, many applied probabilistic models, in teletraffic and insurance modelling for instance, frequently produce distributions with heavy tails and even infinite variance. Regularly varying distributions underlying some of these models fit various data sets particularly well (see Embrechts et al.
\cite{Em97} for examples of financial/actuarial data fitting such a hypothesis).

We consider a stationary sequence of $\mathbb{R}^{d}$ valued random
vectors $(X_n)_{n\geq 1}$ and its accompanying sequence of partial
sums $S_n=X_1+\cdots +X_n,\ {n\geq 1}$. If the $X_{n}$ are i.i.d. and regularly varying with index $\alpha \in (0,2)$, then
\begin{equation}\label{e:first}
 \frac{S_{n}-b_{n}'}{a_{n}'} \dto S_{\alpha} \quad \textrm{in} \ \mathbb{R}^{d},
\end{equation}
for some sequences $a_{n}'>0$ and $b_{n}'$ and some non-degenerate $\alpha$--stable random vector $S_{\alpha}$, see Rva\v{c}eva~\cite{Rv62}
(the univariate result goes back to Gnedenko and Kolmogorov). Weakly dependent sequences,
satisfying  strong mixing condition for instance, can exhibit a very similar behavior.
From the large literature on this phenomenon we refer here to Durrett and Resnick~\cite{Durrett78}, Davis~\cite{Da83}, Denker and Jakubowski~\cite{DeJa89}, Avram and Taqqu~\cite{AvTa92}, Davis and Hsing~\cite{DaHs95}
and Bartkiewicz et al.~\cite{BaJaMiWi09} in the one-dimensional case, and Phillip~\cite{Ph80},\,\cite{Ph86}, Jakubowski and Kobus~\cite{JaKo89} and Davis and Mikosch~\cite{DaMi98} in the multi-dimensional case.

In this paper we are interested in the functional generalization of (\ref{e:first}).
For infinite variance i.i.d.
regularly varying sequences $(X_{n})$ in the one-dimensional case
functional limit theorem was established in Skorohod~\cite{Sk57}. A very
readable proof of this result in the multivariate case can be found in Resnick~\cite{Resnick07} using Skorohod's $J_{1}$ topology on $D([0,1], \mathbb{R}^{d})$.
Tyran-Kami\'{n}ska~\cite{Ty10} recently studied the problem
for a more general class of weakly dependent stationary sequences
using the same topology. However, this choice of topology excludes many processes used in applications.
To study such models we are forced to use a weaker topology.

Our main theorem extends
the main result in Basrak et al.~\cite{BKS} to the multivariate setting. In~\cite{BKS} a functional limit theorem has been obtained for stationary, regularly varying sequence of dependent random variables
for which clusters of high-threshold excesses can be broken down into
asymptotically independent blocks, using the Skorohod's $M_{1}$ topology. Direct generalization of this result to random vectors fails in standard $M_1$ topology, as illustrated by an example in Section~\ref{S:examples}.
It turns out that the limit theorem still holds but in  the weak $M_{1}$ topology.
This topology is strictly weaker than the standard $M_{1}$ topology on $D([0,1], \mathbb{R}^{d})$ for $d\geq 2$ (cf. Whitt~\cite{Whitt02}).
Our main result seems to be the first generic functional limit theorem which holds in
weak $M_1$
topology, but fails in other more frequently used topologies on $D([0,1], \mathbb{R}^{d})$.

\section{Assumptions}
\label{S:statpoint}

\subsection{Regular variation}
\label{SS:statpoint:tail}

Denote $\EE=[-\infty, \infty]^{d} \setminus \{ 0 \}$. The space
$\EE$ is equipped with the topology in which a set $B \subset \EE$
has compact closure if and only if it is bounded away from zero,
that is, if there exists $u > 0$ such that $B \subset \EE_u = \{ x
\in \EE : \|x\| >u \}$. Denote by $C_{K}^{+}(\EE)$ the class of all
nonnegative, continuous functions on $\EE$ with compact support.

We say that a strictly stationary process $(X_{n})_{n \in
\mathbb{Z}}$ is \emph{(jointly) regularly varying} with index
$\alpha \in (0,\infty)$ if for any nonnegative integer $k$ the
$kd$-dimensional random vector $X = (X_{1}, \ldots , X_{k})$ is
multivariate regularly varying with index $\alpha$, i.e.\ there
exists a random vector $\Theta$ on the unit sphere
$\mathbb{S}^{kd-1} = \{ x \in \mathbb{R}^{kd} : \|x\|=1 \}$ such
that for every $u \in (0,\infty)$ and as $x \to \infty$,
 \begin{equation}\label{e:regvar1}
   \frac{\Pr(\|X\| > ux,\,X / \| X \| \in \cdot \, )}{\Pr(\| X \| >x)}
    \wto u^{-\alpha} \Pr( \Theta \in \cdot \,),
 \end{equation}
the arrow ``$\wto$'' denoting weak convergence of finite measures.

Theorem~2.1 in Basrak and Segers \cite{BaSe} provides a convenient
characterization of joint regular variation: it is necessary and
sufficient that there exists a process $(Y_n)_{n \in \mathbb{Z}}$
with $\Pr(\|Y_0\| > y) = y^{-\alpha}$ for $y \ge 1$ such that as $x
\to \infty$,
\begin{equation}\label{e:tailprocess}
  \bigl( (x^{-1}\ X_n)_{n \in \ZZ} \, \big| \, \|X_0\| > x \bigr)
  \fidi (Y_n)_{n \in \ZZ},
\end{equation}
where ``$\fidi$'' denotes convergence of finite-dimensional
distributions. The process $(Y_{n})_{n \in \mathbb{Z}}$ is called
the \emph{tail process} of $(X_{n})_{n \in \mathbb{Z}}$. Writing
$\Theta_n = Y_n / \|Y_0\|$ for $n \in \ZZ$, we also have
\begin{equation}\label{e:spectailprocess}
  \bigl( (\|X_0\|^{-1}X_n)_{n \in \ZZ} \, \big| \, \|X_0\| > x \bigr)
  \fidi (\Theta_n)_{n \in \ZZ},
\end{equation}
see Corollary 3.2 in \cite{BaSe}. The process $(\Theta_n)_{n \in
\ZZ}$ is independent of $\|Y_0\|$ and is called the \emph{spectral
(tail) process} of $(X_n)_{n \in \ZZ}$. The law of $\Theta_0 = Y_0 /
\|Y_0\| \in \mathbb{S}^{d-1}$ is the spectral measure of the common
distribution of the random vectors $X_i$. Regular variation of this
distribution can be expressed in terms of vague convergence of
measures on $\EE$ as follows: for $a_n$ as such that
\begin{equation}\label{e:niz}
 n \Pr( \|X_{1}\| > a_{n}) \to 1,
\end{equation}
 as $n \to
\infty$,
\begin{equation}
  \label{e:onedimregvar}
  n \Pr( a_n^{-1} X_i \in \cdot \, ) \vto \mu( \, \cdot \,),
\end{equation}
where the limit $\mu$ is a nonzero Radon measure on $\EE$ that
satisfies $\mu([-\infty,\infty]^{d} \setminus \mathbb{R}^{d})=0$.
Further, the measure $\mu$ satisfies the following scaling property
\begin{equation}
 \label{e:scalingprop}
 \mu(u\,\cdot\,) = u^{-\alpha} \mu (\,\cdot\,),
 \end{equation}
 for every $u>0$, where $\alpha$ is the same sa in relation
 (\ref{e:regvar1}).

\subsection{Point processes and dependence conditions}
\label{SS:statpoint:point}

We define the time-space point processes
\begin{equation}
\label{E:ppspacetime}
  N_{n} = \sum_{i=1}^{n} \delta_{(i / n,\,X_{i} / a_{n})} \qquad \text{ for all $n\in\NN$,}
\end{equation}
with $a_n$ as in \eqref{e:niz}. In this section we find a limit in
distribution for the sequence $(N_n)_n$ in the state space $[0, 1]
\times \EE_u$ for $u > 0$ under appropriate dependence
assumptions. The limit process is a Poisson
superposition of cluster processes, whose distribution is determined
by the law of the tail process $(Y_i)_{i \in \ZZ}$.

To control the dependence in the sequence $(X_n)_{n \in \mathbb{Z}}$
we first have to assume that clusters of large values of $\|X_{n}\|$
have finite mean size, roughly speaking.

\begin{cond}
\label{c:anticluster} There exists a positive integer sequence
$(r_{n})_{n \in \mathbb{N}}$ such that $r_{n} \to \infty $ and
$r_{n} / n \to 0$ as $n \to \infty$ and such that for every $u > 0$,
\begin{equation}
\label{e:anticluster}
  \lim_{m \to \infty} \limsup_{n \to \infty}
  \Pr \biggl( \max_{m \le |i| \le r_{n}} \|X_{i}\| > ua_{n}\,\bigg|\,\|X_{0}\|>ua_{n} \biggr) = 0.
\end{equation}
\end{cond}

Put $M_{1,n} = \max \{ \|X_{i}\| : i=1, \ldots , n \}$ for $n \in
\NN$. In Proposition~4.2 in \cite{BaSe}, it has been shown that
under Condition~\ref{c:anticluster} the following
holds
\begin{eqnarray}
   \theta  &=& \lim_{r \to \infty} \lim_{x \to \infty} \Pr \bigl(M_{1,r} \le x \, \big| \, \|X_{0}\|>x \bigr)\nonumber \\
   & =& \Pr ({\textstyle\sup_{i\ge 1}} \|Y_{i}\| \le 1) = \Pr ({\textstyle\sup_{i\le -1}} \|Y_{i}\| \le 1)>0. \label{E:theta:spectral}
\end{eqnarray}
 By Remark~4.7 in \cite{BaSe}, alternative
expressions for $\theta$ in \eqref{E:theta:spectral} are
\begin{eqnarray*}
    \theta
    &=& \int_1^\infty
    \Pr \biggl( \sup_{i \geq 1} \norm{\Theta_i}^\alpha \leq y^{-\alpha} \biggr)
    \, \rmd(-y^{-\alpha}) \\
    &=& \E \biggl[ \max \biggl( 1 - \sup_{i \geq 1} \norm{\Theta_i}^\alpha, 0 \biggr) \biggr]
    = \E \biggl[ \sup_{i \geq 0} \norm{\Theta_i}^\alpha
    - \sup_{i \geq 1} \norm{\Theta_i}^\alpha \biggr].
\end{eqnarray*}
Moreover we have $\Pr( \lim_{|n| \to \infty} \norm{Y_n} =
0 ) = 1$. Also, for every $u \in (0, \infty)$
\begin{equation}
\label{E:runsblocks}
    \Pr(M_{1,r_n} \leq a_n u \mid \norm{X_0} > a_n u)
    = \frac{\Pr(M_{1,r_n} > a_n u)}{r_n \Pr(\norm{X_0} > a_n u)} + o(1)
    \to \theta
\end{equation}
as $n \to \infty$.

In the sequel, the  point processes
\[
  \sum_{i=1}^{r_n} \delta_{(a_n u)^{-1} X_i} \quad \text{conditionally on} \quad M_{1,r_n} > a_n u\,,
\]
are called \emph{cluster processes}. We use them to describe a cluster of
extremes occurring in a relatively short time
span. Theorem~4.3 in \cite{BaSe} yields the weak convergence of the
sequence of cluster processes in the state space $\EE$:
\begin{equation}
\label{E:clusterprocess}
    \biggl( \sum_{i=1}^{r_n} \delta_{(a_n u)^{-1} X_i} \, \bigg| \, M_{1,r_n} > a_n u \biggr)
    \dto \biggl( \sum_{j \in \mathbb{Z}} \delta_{Y_j} \, \bigg| \, \sup_{i \le -1} \|Y_i\| \le 1 \biggr).
\end{equation}

Note that since $\|Y_n\| \to 0$ almost surely as $|n| \to \infty$,
the point process $\sum_n \delta_{Y_n}$ is well-defined in $\EE$. By
\eqref{E:theta:spectral}, the probability of the conditioning event
on the right-hand side of \eqref{E:clusterprocess} is nonzero.

To establish convergence of $N_n$ in \eqref{E:ppspacetime}, we need
to impose a certain mixing condition called $\mathcal{A}'(a_n)$
which is slightly stronger than the condition~$\mathcal{A}(a_n)$
introduced in Davis and Hsing~\cite{DaHs95} and Davis and
Mikosch~\cite{DaMi98}.

\begin{cond}[$\mathcal{A}'(a_{n})$]
\label{c:mixcond} There exists a sequence of positive integers
$(r_{n})_{n}$ such that $r_{n} \to \infty $ and $r_{n} / n \to 0$ as
$n \to \infty$ and such that for every $f \in C_{K}^{+}([0,1] \times
\EE)$, denoting $k_{n} = \lfloor n / r_{n} \rfloor$, as $n \to
\infty$,
\begin{equation}\label{e:mixcon}
 \Exp \biggl[ \exp \biggl\{ - \sum_{i=1}^{n} f \biggl(\frac{i}{n}, \frac{X_{i}}{a_{n}}
 \biggr) \biggr\} \biggr]
 - \prod_{k=1}^{k_{n}} \Exp \biggl[ \exp \biggl\{ - \sum_{i=1}^{r_{n}} f \biggl(\frac{kr_{n}}{n}, \frac{X_{i}}{a_{n}} \biggr) \biggr\} \biggr] \to 0.
\end{equation}
\end{cond}

It can be shown that Condition~\ref{c:mixcond} is implied by the
strong mixing property, see Krizmani\'c~\cite{Kr10}.

The  proof of Theorem 2.3 in  Basrak et al.~\cite{BKS}
carries over to the multivariate case with some straightforward adjustments.
Hence, we obtain the following result describing exceedences in the sequence $(X_n)$
outside of the ball of radius $u$ around the origin.

\begin{theorem}
\label{T:pointprocess:complete} Assume that Conditions~\ref{c:anticluster}
and \ref{c:mixcond} hold
for the same sequence $(r_n)$, then for every $u \in (0, \infty)$ and as
$n \to \infty$,
\[
    N_n \bigg|_{[0, 1] \times \EE_u}\, \dto N^{(u)}
    = \sum_i \sum_j \delta_{(T^{(u)}_i, u Z_{ij})} \bigg|_{[0, 1] \times \EE_u}\,,
\]
in $[0, 1] \times \EE_u$ and
\begin{enumerate}
\item $\sum_i \delta_{T^{(u)}_i}$ is a homogeneous Poisson process on $[0, 1]$ with intensity $\theta u^{-\alpha}$,
\item $(\sum_j \delta_{Z_{ij}})_i$ is an i.i.d.\ sequence of point processes in $\EE$, independent of $\sum_i \delta_{T^{(u)}_i}$, and with common distribution equal to the weak limit in \eqref{E:clusterprocess}.
\end{enumerate}
\end{theorem}

\section{Functional limit theorem}
\label{S:flt}

In the main result in the article, we show the convergence of the partial sum
process $V_n$ to a stable L\'evy process in the space
$D([0, 1], \mathbb{R}^{d})$ equipped with Skorohod's weak $M_1$
topology. As in the one dimensional case (cf. Basrak et al.~\cite{BKS}) we first represent the partial sum process $V_n$ as
the image of the time-space point process $N_n$ in
\eqref{E:ppspacetime} under a certain summation functional. Then, since this
summation functional enjoys the right continuity properties, by an application of the continuous mapping theorem we transfer
the weak convergence of $N_n$ in
Theorem~\ref{T:pointprocess:complete} to weak convergence
of $V_n$.

\subsection{The weak $M_1$ topology}
\label{SS:M1}

For $a=(a^{1}, \ldots, a^{d}), b=(b^{1}, \ldots, b^{d}) \in
\mathbb{R}^{d}$, let $[[a,b]]$ be the product segment, i.e.\
$$[[a,b]]=[a^{1},b^{1}] \times [a^{2},b^{2}]
\times \dots \times [a^{d},b^{d}].$$
 For $x \in D([0,1],
\mathbb{R}^{d})$ the \emph{completed graph} of $x$ is the set
\[
  G_{x}
  = \{ (t,z) \in [0,1] \times \mathbb{R}^{d} : z \in [[x(t-), x(t)]]\},
\]
where $x(t-)$ is the left limit of $x$ at $t$. We define an
\emph{order} on the graph $G_{x}$ by saying that $(t_{1},z_{1}) \le
(t_{2},z_{2})$ if either (i) $t_{1} < t_{2}$ or (ii) $t_{1} = t_{2}$
and $|x^{j}(t_{1}-) - z^{j}_{1}| \le |x^{j}(t_{2}-) - z^{j}_{2}|$
for all $j=1,\ldots,d$. Clearly, the relation $\leqslant$ induces only a partial
order on the graph $G_{x}$. A \emph{weak parametric representation}
of the graph $G_{x}$ is a continuous nondecreasing function $(r,u)$
mapping $[0,1]$ into $G_{x}$, with $r \in C([0,1],[0,1])$ being the
time component and $u=(u^{1},\ldots, u^{d}) \in C([0,1],
\mathbb{R}^{d})$ being the spatial component, such that $r(0)=0,
r(1)=1$ and $u(1)=x(1)$. Let $\Pi_{w}(x)$ denote the set of weak
parametric representations of the graph $G_{x}$. For $x_{1},x_{2}
\in D([0,1], \mathbb{R}^{d})$ define
\[
  d_{w}(x_{1},x_{2})
  = \inf \{ \|r_{1}-r_{2}\|_{[0,1]} \vee \|u_{1}-u_{2}\|_{[0,1]} : (r_{i},u_{i}) \in \Pi_{w}(x_{i}), i=1,2 \},
\]
where $\|x\|_{[0,1]} = \sup \{ \|x(t)\| : t \in [0,1] \}$. Now we
say that $x_{n} \to x$ in $D([0,1], \mathbb{R}^{d})$ for a sequence
$(x_{n})$ in the weak Skorohod's $M_{1}$ (or shortly $WM_{1}$)
topology if $d_{w}(x_{n},x)\to 0$ as $n \to \infty$. The $WM_{1}$
topology is weaker than the standard $M_{1}$ topology on $D([0,1],
\mathbb{R}^{d})$. Note, however that for $d=1$ two topologies coincide. The $WM_{1}$ topology
coincides with the topology induced by the metric
$$ d_{p}(x_{1},x_{2})=\max \{ d_{M_{1}}(x_{1}^{j},x_{2}^{j}) :
j=1,\ldots,d\}$$
 for $x_{i}=(x_{i}^{1}, \ldots, x_{i}^{d}) \in D([0,1],
 \mathbb{R}^{d})$ and $i=1,2$ (here $d_{M_{1}}$ denotes the standard Skorohod's
 $M_{1}$ metric on $D([0,1],\mathbb{R})$). The metric $d_{p}$ induces the product topology on $D([0,1], \mathbb{R}^{d})$.
For detailed discussion of the weak $M_{1}$ topology we refer to
Whitt~\cite{Whitt02}.

\begin{figure}\label{Slika:graf}
\centering
\includegraphics[width=12cm]{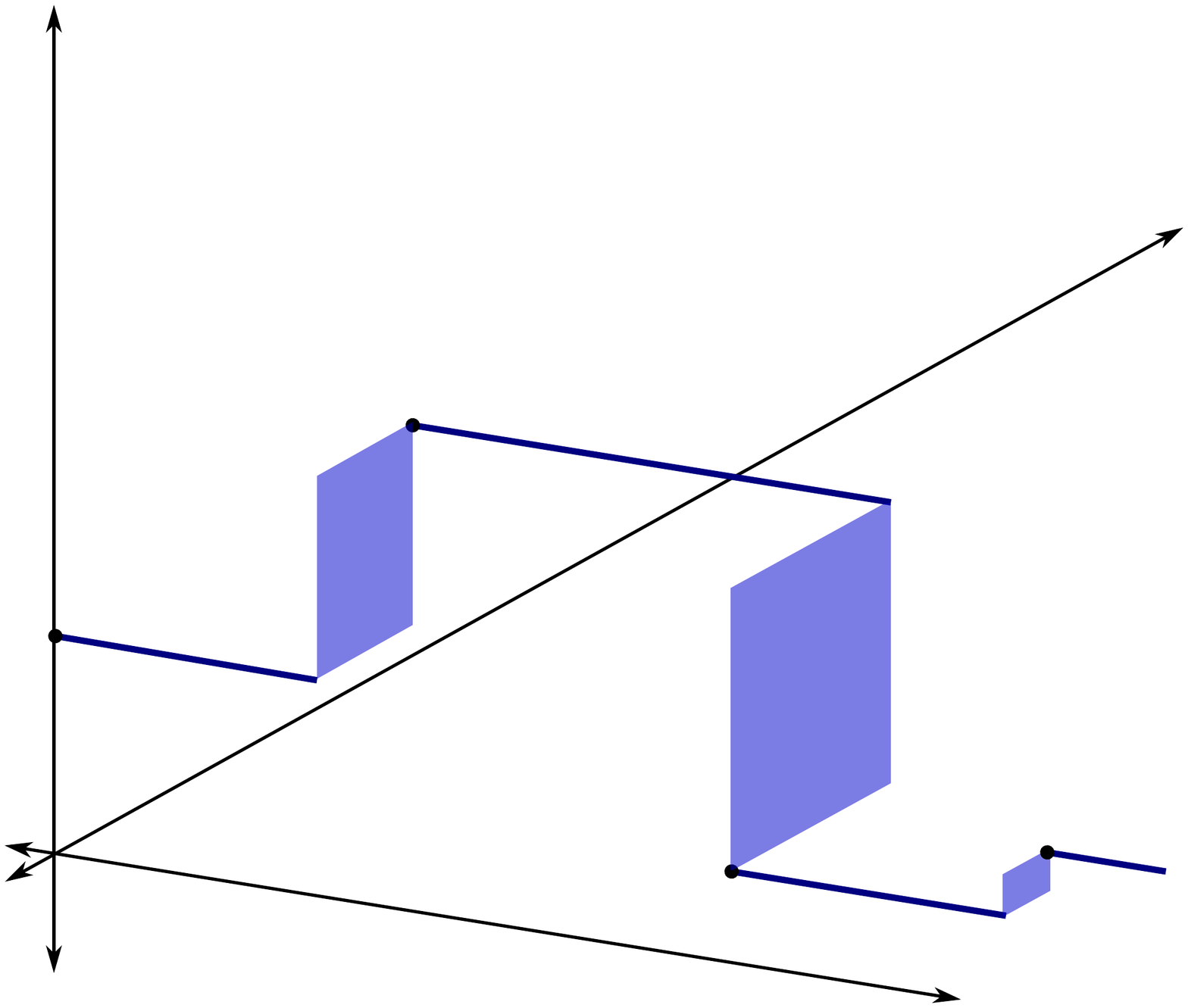}\\
\caption{Completed graph of a c\'adl\'ag function in $\RR^2$. }
\end{figure}

\subsection{Continuity of summation functional}
\label{SS:sumfunct}

Fix $0 < v < u < \infty$. The proof of our main theorem depends on
the continuity properties of the summation functional
\[
  \psi^{(u)} \colon \mathbf{M}_{p}([0,1] \times \EE_{v}) \to
  D([0,1], \mathbb{R}^{d})
\]
defined by
\[
  \psi^{(u)} \bigl( \som_{i}\delta_{(t_{i},\,(x_{i}^{1},\ldots,x_{i}^{d}))} \bigr) (t)
  = \Big( \sum_{t_{i} \le t} x_{i}^{j} \,1_{\{u < |x_i^{j}| < \infty\}} \Big)_{j=1,\ldots,d}, \qquad t \in [0, 1].
\]
Observe that $\psi^{(u)}$ is well defined because $[0,1] \times
\EE_{u}$ is a relatively compact subset of $[0,1] \times \EE_{v}$.
The space $\mathbf{M}_p$ of Radon point measures is equipped with
the vague topology and $D([0, 1], \mathbb{R}^{d})$ is equipped with
the weak $M_1$ topology.

We will show that $\psi^{(u)}$ is continuous on the set $\Lambda =
\Lambda_{1} \cap \Lambda_{2}$, where
\begin{multline*}
  \Lambda_{1} =
  \{ \eta \in \mathbf{M}_{p}([0,1] \times \EE_{v}) :
    \eta ( \{0,1 \} \times \EE_{u}) = 0 \ \textrm{and} \\[0.3em]
     \eta ([0,1] \times \{ x=(x^{1},\ldots,x^{d}) : |x^{i}| \in \{u, \infty \} \ \textrm{for some} \ i \}) =0 \}, \\[1em]
  \shoveleft \Lambda_{2} =
  \{ \eta \in \mathbf{M}_{p}([0,1] \times \EE_{v}) :
     \text{for all $t \in [0,1]$}, \
    \eta ( \{ t \} \times \prod_{j=1}^{d}A_{j}) = 0
   \ \textrm{for all except} \\[0.3em] \textrm{at most one set} \ \prod_{j=1}^{d}A_{j} \ \textrm{where} \ A_{j}=(0,+\infty]
   \ \textrm{or} \ A_{j}=[-\infty,0)
   \}.
\end{multline*}

\begin{lemma}
\label{l:prob1} Assume that with probability one, the tail process
$(Y_{i})_{i \in \mathbb{Z}}$ in \eqref{e:tailprocess} satisfies the
property that for every $j=1,\ldots,d$, $(Y_{i}^{j})_{i \in
\mathbb{Z}}$ has no two values of the opposite sign. Then $ \Pr (
N^{(v)} \in \Lambda ) = 1$.
\end{lemma}

\begin{proof}
Recall that we can write $Y_{i}^{j} = \|Y_0\|\,\Theta_i^{j}$ for all indices $i$ and coordinates $j$, see \eqref{e:spectailprocess}.
Observe that  random variable $ \|Y_0\|$  is continuous and independent of all
$\Theta_i^{j}$'s to obtain
$$ \Pr( |Y_{i}^{j}|=u/v \ \textrm{for some} \ j=1,\ldots,d )=0.$$

Therefore, in light of Theorem
\ref{T:pointprocess:complete} it holds that
\begin{eqnarray*}
  \Pr \Big( \sum_{j}\delta_{vZ_{ij}} (\{x=(x^{1},\ldots,x^{d}) : |x^{k}|=u \ \textrm{for some} \ k\})=0 \Big) &&
  \\[0.2em]
   &\hspace*{-52em} =& \hspace*{-25.3em} \Pr \Big( \sum_{j}\delta_{Z_{ij}} (\{x: |x^{k}|=u/v \ \textrm{for some} \
   k\} \Big)=0) \\[0.2em]
   &\hspace*{-52em} =& \hspace*{-25.3em} \Pr \Big( \sum_{j}\delta_{Y_{j}} (\{x=(x^{1},\ldots,x^{d}) :
    |x^{k}|=u/v \ \textrm{for some} \ k\})=0\,\Big|\,\sup_{i}\|Y_{i}\| \leqslant -1
    \Big)\\[0.2em]
    &\hspace*{-52em} =& \hspace*{-25.3em} 1.
\end{eqnarray*}
Hence
$$ \Pr (N^{v} ([0,1] \times \{ x=(x^{1},\ldots,x^{d}) : |x^{k}|=u \
\textrm{for some} \ k \})=0)=1.$$ We obtain the same result if above
we replace $u$ by $+\infty$, and this together with the fact that
$\Pr( \sum_{i}\delta_{T_{i}^{(v)}} (\{0,1\}) = 0 )=1$ implies $\Pr(
N^{(v)} \in \Lambda_{1})=1$.

Second, the assumption that with probability one $(Y_{i}^{j})_{i \in
\mathbb{Z}}$ has no two values of the opposite sign for every
$j=1,\ldots,d$ yields $\Pr(N^{(v)} \in \Lambda_{2})=1$.\qed
\end{proof}

\begin{lemma}
\label{L:contsf} The summation functional $\psi^{(u)} \colon
\mathbf{M}_{p}([0,1] \times \EE_{v}) \to D([0,1], \mathbb{R}^{d})$
is continuous on the set $\Lambda$, when $D([0,1], \mathbb{R}^{d})$
is endowed with Skorohod's weak $M_{1}$ topology.
\end{lemma}

\begin{proof}
Suppose that $\eta_{n} \vto \eta$ in $\mathbf{M}_p([0,1] \times
\EE_{v})$ for some $\eta \in \Lambda$. We need to show that
$\psi^{(u)}(\eta_n) \to \psi^{(u)}(\eta)$ in $D([0, 1],
\mathbb{R}^{d})$ according to the $WM_1$ topology. By
Theorem~12.5.2 in Whitt~\cite{Whitt02}, it suffices to prove that,
as $n \to \infty$,
$$ d_{p}(\psi^{(u)}(\eta_{n}), \psi^{(u)}(\eta)) =
\max_{j=1,\ldots,d}d_{M_{1}}(\psi^{(u)\,j}(\eta_{n}),
\psi^{(u)\,j}(\eta)) \to 0,$$
 where $\psi^{(u)}(\xi)=(\psi^{(u)\,j}(\xi))_{j=1,\ldots,d}$ for
 $\xi \in \mathbf{M}_{p}([0,1] \times \EE_{v})$.

 Now one can follow, with small modifications, the lines in the proof of Lemma~3.2 in Basrak et al.~\cite{BKS} to obtain
$d_{M_{1}}(\psi^{(u)\,j}(\eta_{n}), \psi^{(u)\,j}(\eta)) \to 0$ as
$n \to \infty$. Therefore $d_{p}(\psi^{(u)}(\eta_{n}),
\psi^{(u)}(\eta)) \to 0$ as $n \to \infty$, and we conclude that
$\psi^{(u)}$ is continuous at $\eta$.\qed
\end{proof}

\subsection{Main theorem}
\label{SS:main}

Let $(X_n)_n$ be a strictly stationary sequence of random vectors,
jointly regularly varying with index $\alpha \in (0, 2)$ and tail
process $(Y_i)_{i \in \ZZ}$. The main theorem gives conditions
under which its partial sum process satisfies a nonstandard
functional limit theorem with a non-Gaussian $\alpha$--stable
L\'{e}vy process as a limit. Recall that the distribution of a
L\'{e}vy process $V(\,\cdot\,)$ is characterized by its
\emph{characteristic triple}, i.e.\ the characteristic triple of the
infinitely divisible distribution of $V(1)$. The characteristic
function of $V(1)$ and the characteristic triple $(A, \nu, b)$ are
related in the following way:
\[
  \E [e^{i \langle z,V(1) \rangle}] = \exp \biggl( -\frac{1}{2} \langle z, Az \rangle + i \langle b, z \rangle
  + \int_{\mathbb{R}^{d}} \bigl( e^{i \langle z, x \rangle}-1- i \langle z, x \rangle 1_{\{ \|x\|_{2} \leqslant 1 \}} \bigr)\,\nu(\rmd x) \biggr)
\]
for $z \in \mathbb{R}^{d}$, where $\langle x, y \rangle =
\sum_{i=1}^{d}x^{i}y^{i}$ and
$\|x\|_{2}=\sqrt{\sum_{i=1}^{d}(x^{i})^{2}}$ for
$x=(x^{1},\ldots,x^{d}),\,y=(y^{1},\ldots,y^{d}) \in
\mathbb{R}^{d}$. Here $A$ is a symmetric nonnegative-definite $d
\times d$ matrix, $\nu$ is a measure on $\mathbb{R}^{d}$ satisfying
\[
  \nu ( \{0\})=0 \qquad \text{and} \qquad \int_{\mathbb{R}^{d}}(\|x\|_{2}^{2} \wedge 1)\,\nu(\rmd x) < \infty,
\]
(that is, $\nu$ is a L\'{e}vy measure), and $b \in \mathbb{R}^{d}$.
For a textbook treatment of L\'{e}vy processes we refer to
Bertoin~\cite{Bertoin96} and Sato~\cite{Sato99}. The description of
the characteristic triple of the limit process will be in terms of the
measures $\nu^{(u)}$ ($u > 0$) on $\EE$ defined by
\begin{equation}
\label{E:nuu}
 \begin{array}{rl}
 \nu^{(u)}((x, y]) & = \displaystyle u^{-\alpha} \, \Pr \biggl( u \sum_{i \ge 0} \big( Y_i^{j} \, 1_{\{|Y_i^{j}| > 1\}} \big)_{j=1,\ldots,d} \in (x,y], \, \sup_{i \le -1} \|Y_i\| \le 1
  \biggr),
  \end{array}
\end{equation}
for $x=(x^{1},\ldots,x^{d}),\,y=(y^{1},\ldots,y^{d}) \in \EE$ such
that $(x,y]=(x^{1},y^{1}] \times \dots \times (x^{d},y^{d}]$ is
bounded away from zero.

Our main result considers the limiting behavior of the partial sum stochastic process
\begin{equation}
 \label{e:psp}
  V_{n}(t) =
  \sum_{k=1}^{\floor{nt}} \frac{X_{k}}{a_{n}} -
  \lfloor nt \rfloor \E \bigg( \bigg( \frac{X_{1}^{j}}{a_{n}} 1_{ \big\{ \frac{|X_{1}^{j}|}{a_{n}} \le 1 \big\} } \bigg)_{j=1,\ldots,d} \bigg), \quad t \in [0,1],
\end{equation}

It turns out that in the case $\alpha \in [1, 2)$, we need to assume  that the
contributions of the smaller increments to this partial sum process is
close to their expectation.

\begin{cond}
\label{c:step6cond} For all $\delta > 0$,
\[
  \lim_{u \downarrow 0} \limsup_{n \to \infty} \Pr \bigg[
  \max_{1 \le k \le n}  \bigg\| \sum_{i=1}^{k} \bigg( \frac{X_{i}^{j}}{a_{n}}
  1_{ \big\{ \frac{|X_{i}^{j}|}{a_{n}} \le u \big\} } -  \E \bigg( \frac{X_{i}^{j}}{a_{n}}
  1_{ \big\{ \frac{|X_{i}^{j}|}{a_{n}} \le u \big\} } \bigg) \bigg)_{j=1,\ldots,d} \bigg\| > \delta
  \bigg]=0.
\]
\end{cond}

\begin{theorem}
\label{t:2} Let $(X_{n})_{n \in \mathbb{N}}$ be a strictly
stationary sequence of random vectors, jointly regularly varying
with index $\alpha\in(0,2)$, and such that the conditions of Theorem~\ref{T:pointprocess:complete}
and Lemma~\ref{l:prob1} hold.
 If $1 \le \alpha < 2$,  suppose further that Condition~\ref{c:step6cond} holds.
Then
\[
  V_{n} \dto V, \qquad n \to \infty,
\]
in $D([0,1], \mathbb{R}^{d})$ endowed with the weak $M_{1}$
topology, where $V(\,\cdot\,)$ is an $\alpha$--stable L\'{e}vy
process.
\end{theorem}

\begin{remark}
The condition about the tail sequence in Lemma~\ref{l:prob1} and Theorem~\ref{t:2} although restrictive, holds trivially
for all random vectors with nonnegative components. In such a case,  random variables
$(Y_{i}^{j})_{i
 \in \ZZ}$ are all nonnegative a.s. as well.  Therefore, they
 cannot have values with the opposite sign for any $j=1,\ldots,d$.
This means that the limit theorem applies
directly for many sequences appearing in applications, such as claim or file sizes in insurance or teletraffic modelling for instance.
\end{remark}

\begin{remark}\label{R:rhomix}
In general technical Condition~\ref{c:step6cond} is particularly difficult to check.
One  sufficient condition for
condition~\ref{c:step6cond} to hold can be given using
 the notion of
$\rho$-mixing. Recall that a strictly stationary sequence
$(Z_{i})_{i \in \mathbb{Z}}$ is \emph{$\rho$-mixing} if
$$  \rho_{n} = \sup \{ |\operatorname{corr} (U, V)| :  U \in
   L^{2}(\mathcal{F}_{-\infty}^{0}),\,V \in
   L^{2}(\mathcal{F}_{n}^{\infty}) \} \to 0 \quad \textrm{as} \ n \to 0.$$
Note that $\rho$-mixing implies strong mixing, whereas the converse
in general does not hold, see Bradley~\cite{Bradley05}.
Using a slight modification of the proof of Lemma 4.8 in Tyran-Kami\'{n}ska~\cite{Ty10} (see also Corollary~2.1 in Peligrad~\cite{Peligrad99}) one can show the following:
For  a strictly stationary sequence $(X_{n})_{n}$
of regularly varying random vectors with index $\alpha \in [1,2)$,
and a sequence $(a_{n})$ satisfying
$(\ref{e:niz})$, Condition~\ref{c:step6cond} holds if $(X_{n})_{n}$ is
$\rho$-mixing with
\[
  \sum_{j \ge 0} \rho_{2^{j}} < \infty.
\]
This further means that for an $m$--dependent sequence of random vectors,
 Condition~\ref{c:step6cond} always holds.
\end{remark}

\begin{remark}
The characteristic L\'{e}vy triple $(0, \nu, b)$ of the limiting process $V$ in the theorem is
given by the limits  in
\begin{align*}
  \nu^{(u)} &\vto \nu, &\ \  \bigg( \int_{x : u < \|x\| \le 1} x^{j} \, \nu^{(u)}(\rmd x) - \int_{x : u < |x^{j}| \le 1} x^{j} \, \mu(\rmd x) \bigg)_{j=1,\ldots,d} &\to b
\end{align*}
as $u \downarrow 0$, with $\nu^{(u)}$ as in \eqref{E:nuu} and $\mu$
as in \eqref{e:onedimregvar}.
\end{remark}

\begin{proof} [Theorem~\ref{t:2}]
Note that from Theorem~\ref{T:pointprocess:complete} and the fact
that $\|Y_{n}\| \to 0$ almost surely as $|n| \to \infty$, the random
vectors
\[
  u \sum_{j} \Big( Z_{ij}^{k}1_{\{ |Z_{ij}^{k}|>1 \}} \Big)_{k=1,\ldots,d}
\]
are i.i.d.\ and almost surely finite. Define
 $$ \widehat{N}^{(u)} = \sum_{i} \delta_{(T_{i}^{(u)},\,u\sum_{j} (Z_{ij}^{k}1_{\{ |Z_{ij}^{k}|>1
 \}})_{k=1,\ldots,d})}.$$
Then by Proposition~5.3 in Resnick~\cite{Resnick07},
$\widehat{N}^{(u)}$ is a Poisson process (or a Poisson random
measure) with mean measure  \begin{equation}\label{e:prodmeas}
  \theta u^{-\alpha} \lambda \times F^{(u)},
\end{equation}
where $\lambda$ is the Lebesgue measure and $F^{(u)}$ is the
distribution of the random vector $u \sum_{j} (Z_{1j}^{k}1_{\{
|Z_{1j}^{k}|>1 \}})_{k=1,\ldots,d}$. But for $0 \le s < t \le 1$ and
$x,y \in \EE$ such that $(x,y]$ is bounded away from zero, using the
fact that the distribution of $\sum_{j}\delta_{Z_{1j}}$ is equal to
the one of $\sum_{j}\delta_{Y_{j}}$ conditionally on the event $\{
\sup_{i \le -1}\|Y_{i}\| \le 1 \}$, after standard calculations we obtain
  \begin{equation*}
 \theta u^{-\alpha} \lambda \times F^{(u)} ([s,t] \times
(x,y])  =  \lambda \times \nu^{(u)}([s,t] \times (x, y]).
\end{equation*}
Thus the mean measure in \eqref{e:prodmeas} is equal to $\lambda
\times \nu^{(u)}$.

Consider now $0<v<u$ and
$$  \psi^{(u)} (N_{n}\,|\,_{[0,1] \times \EE_{u}}) (\,\cdot\,)
  = \psi^{(u)} (N_{n}\,|\,_{[0,1] \times \EE_{v}}) (\,\cdot\,)
  = \sum_{i/n \le \, \cdot} \Big( \frac{X_{i}^{k}}{a_{n}} 1_{ \big\{ \frac{|X_{i}^{k}|}{a_{n}} > u
    \big\} } \Big)_{k=1,\ldots,d},$$
which by Lemma~\ref{L:contsf} converges in distribution in $D[0,1]$
under the $WM_{1}$ topology to
$$
\psi^{(u)} (N^{(v)})(\,\cdot\,)
\eind \psi^{(u)} (N^{(v)}\,|\,_{[0,1] \times \EE_{u}})(\,\cdot\,).
$$
However, by the definition of the process $N^{(u)}$ in
Theorem~\ref{T:pointprocess:complete} it holds that
$$ N^{(u)} \eind
N^{(v)} \bigg|_{[0, 1] \times \EE_u}\,, $$ for every $v\in (0,u)$.
Therefore the last expression above is equal in distribution to
$$
\psi^{(u)} (N^{(u)})(\,\cdot\,)
 = \sum_{T_{i}^{(u)} \le \, \cdot}
   \sum_{j}u (Z_{ij}^{k}1_{ \{ |Z_{ij}^{k}| > 1 \} })_{k=1,\ldots,d}.
$$
 But since
  $\psi^{(u)}(N^{(u)}) = \psi^{(u)} (\widehat{N}^{(u)})\,\eind\,\psi^{(u)} (\widetilde{N}^{(u)})$,
 where
 $$ \widetilde{N}^{(u)} = \sum_{i} \delta_{(T_{i},\,K_{i}^{(u)})}
 $$
 is a Poisson process with mean measure $\lambda \times \nu^{(u)}$,
 we obtain
 $$ \sum_{i = 1}^{\lfloor n \, \cdot \, \rfloor} \Big( \frac{X_{i}^{k}}{a_{n}} 1_{ \big\{ \frac{|X_{i}^{k}|}{a_{n}} > u
    \big\} }\Big)_{k=1,\ldots,d} \dto \sum_{T_{i} \le \, \cdot} K_{i}^{(u)}, \quad \text{as} \ n \to \infty,$$
 in $D([0,1], \mathbb{R}^{d})$ under the $WM_{1}$ topology. From (\ref{e:onedimregvar}) we
 have, for any $t \in [0,1]$, as $n \to \infty$,
 \begin{eqnarray*}
   \lfloor nt \rfloor \E \bigg( \bigg( \frac{X_{1}^{k}}{a_{n}} \, 1_{ \big\{ u < \frac{|X_{1}^{k}|}{a_{n}} \le 1 \big\} }
    \bigg)_{k=1,\ldots,d} \bigg) & &\\[0.8em]
    & \hspace*{-30em}= & \hspace*{-14.3em} \frac{\lfloor nt \rfloor}{n} \bigg( \int_{\{x\,:\,u < |x^{k}| \le 1 \}}x^{k}
     n \Pr \bigg( \frac{X_{1}}{a_{n}} \in \rmd x \bigg)  \bigg)_{k=1,\ldots,d} \\[0.8em]
    & \hspace*{-30em} \to & \hspace*{-14.3em} t\,\bigg( \int_{\{x\,:\,u < |x^{k}| \le 1 \}}x^{k} \, \mu(\rmd x) \bigg)_{k=1,\ldots,d}.
 \end{eqnarray*}
 This convergence is uniform in $t$ and hence
 $$ \lfloor n \, \cdot \, \rfloor \E \bigg( \bigg( \frac{X_{1}^{k}}{a_{n}} 1_{ \big\{ u < \frac{|X_{1}^{k}|}{a_{n}} \le 1 \big\} }
    \bigg)_{k=1,\ldots,d} \bigg) \to (\,\cdot\,) \bigg( \int_{\{x\,:\,u < |x^{k}| \le 1 \}}x^{k} \, \mu(\rmd x) \bigg)_{k=1,\ldots,d}$$
 in $D([0,1], \mathbb{R}^{d})$.
Since the latter function is continuous, applying an analogue of
Corollary~12.7.1 in Whitt~\cite{Whitt02} but for the metric $d_{p}$, we obtain, as $n \to
\infty$,
\begin{multline}
\label{e:mainconv}
   V_{n}^{(u)}(\,\cdot\,) = \sum_{i = 1}^{\lfloor n \, \cdot \, \rfloor} \bigg( \frac{X_{i}^{k}}{a_{n}}
    1_{ \bigl\{ \frac{|X_{i}^{k}|}{a_{n}} > u \bigr\} } \bigg)_{k=1,\ldots,d} - \lfloor n \,\cdot \, \rfloor
    \E \biggl( \bigg( \frac{X_{1}^{k}}{a_{n}}
    1_{ \bigl\{ u < \frac{|X_{1}^{k}|}{a_{n}} \le 1 \bigr\} }
    \bigg)_{k=1,\ldots,d} \biggr) \\[0.8em]
    \dto  V^{(u)}(\,\cdot\,) := \sum_{T_{i} \le \, \cdot}
   K_{i}^{(u)} - (\,\cdot\,) \bigg( \int_{\{x\,:\,u < |x^{k}| \le 1 \}}x^{k} \, \mu(\rmd x) \bigg)_{k=1,\ldots,d}.
\end{multline}
The limit (\ref{e:mainconv}) can be rewritten as
\begin{multline*}
  \sum_{T_{i} \le \,\cdot}
  K_{i}^{(u)} - (\,\cdot\,) \bigg( \int_{\{x\,:\,u < \|x\| \le 1 \}}x^{k} \, \nu^{(u)}(\rmd x) \bigg)_{k=1,\ldots,d} \\
  + (\,\cdot\,) \biggl( \int_{\{x\,:\,u < \|x\| \le 1 \}}x^{k} \, \nu^{(u)}(\rmd x)
  -  \int_{\{x\,:\,u < |x^{k}| \le 1 \}}x^{k} \, \mu(\rmd x)\biggr)_{k=1,\ldots,d}.
\end{multline*}
Note that the first two terms, since $\nu^{(u)}(\{ x : \|x\|
\leqslant u \})=0$, represent a L\'{e}vy--Ito representation of the
L\'{e}vy process with characteristic triple $(0, \nu^{(u)}, 0)$, see
Resnick~\cite[p.\ 150]{Resnick07}. The remaining term is just a
linear function of the form $t \mapsto t \, b_{u}$. As a
consequence, the process $V^{(u)}$ is a L\'{e}vy process for each
$u<1$, with characteristic triple $(0, \nu^{(u)}, b_{u})$, where
$$ b_{u} = \bigg( \int_{\{x\,:\,u < \|x\| \le 1 \}}x^{k} \,
    \nu^{(u)}(\rmd x) - \int_{\{x\,:\,u < |x^{k}| \le 1 \}}x^{k} \,
    \mu(\rmd x) \bigg)_{k=1,\ldots,d}.$$

By Proposition~3.3 in Davis and Mikosch~\cite{DaMi98}, for $t=1$,
$V^{(u)}(1)$ converges to an $\alpha$--stable random vector. Hence by Theorem~13.17
in Kallenberg~\cite{Kallenberg97}, there is a L\'{e}vy process
$V(\,\cdot\,)$ such that, as $u \to 0$,
$$ V^{(u)}(\,\cdot\,) \dto V(\,\cdot\,)$$
in $D([0,1], \mathbb{R}^{d})$ with the $WM_{1}$ topology. It has
characteristic triple $(0, \nu, b)$, where $\nu$ is the vague limit
of $\nu^{(u)}$ as $u \to 0$ and $b=\lim_{u \to 0}b_{u}$, see
Theorem~13.14 in~\cite{Kallenberg97}. Since the random vector $V(1)$
has an $\alpha$--stable distribution, it follows that the process
$V(\,\cdot\,)$ is $\alpha$--stable.

If we show that
$$ \lim_{u \downarrow 0} \limsup_{n \to \infty}
   \Pr[d_{p}(V_{n}^{(u)}, V_{n}) > \delta]=0$$
 for any $\delta>0$, then by Theorem~3.5 in Resnick~\cite{Resnick07} we
 will have, as $n \to \infty$,
 $$ V_{n} \dto V$$
 in $D([0,1], \mathbb{R}^{d})$ with the $WM_{1}$ topology. Since the
 metric $d_{p}$ on $D([0,1], \mathbb{R}^{d})$ is bounded above by the uniform metric on
 $D([0,1], \mathbb{R}^{d})$ (see Theorem 12.10.3 in Whitt~\cite{Whitt02}), it suffices to show that
 $$ \lim_{u \downarrow 0} \limsup_{n \to \infty} \Pr \biggl(
 \sup_{0 \le t \le 1} \|V_{n}^{(u)}(t) - V_{n}(t)\| >
 \delta \biggr)=0.$$
 Recalling the definitions, we have
 \begin{equation*}
   \begin{split}
    \lim_{u \downarrow 0} & \limsup_{n \to \infty} \Pr \bigg(
     \sup_{0 \le t \le 1} \|V_{n}^{(u)}(t) - V_{n}(t)\| >  \delta \bigg) \\
    & = \lim_{u \downarrow 0} \limsup_{n \to \infty} \Pr \bigg[
       \max_{1 \le k \le n}  \bigg\| \sum_{i=1}^{k} \bigg( \frac{X_{i}^{j}}{a_{n}}
       1_{ \big\{ \frac{|X_{i}^{j}|}{a_{n}} \le u \big\} } -  \E \bigg( \frac{X_{i}^{j}}{a_{n}}
       1_{ \big\{ \frac{|X_{i}^{j}|}{a_{n}} \le u \big\} } \bigg) \bigg)_{j=1,\ldots,d} \bigg\| > \delta
       \bigg].
  \end{split}
 \end{equation*}
 Therefore we have to show
 \begin{equation}\label{e:slutskycond}
     \lim_{u \downarrow 0} \limsup_{n \to \infty} \Pr \bigg[
       \max_{1 \le k \le n}  \bigg\| \sum_{i=1}^{k} \bigg( \frac{X_{i}^{j}}{a_{n}}
       1_{ \big\{ \frac{|X_{i}^{j}|}{a_{n}} \le u \big\} } -  \E \bigg( \frac{X_{i}^{j}}{a_{n}}
       1_{ \big\{ \frac{|X_{i}^{j}|}{a_{n}} \le u \big\} } \bigg) \bigg)_{j=1,\ldots,d} \bigg\| > \delta
       \bigg]=0.
 \end{equation}
For $\alpha \in [1,2)$ this relation is simply
Condition~\ref{c:step6cond}. Therefore it remains to show
(\ref{e:slutskycond}) for the case when $\alpha \in (0,1)$. Hence
assume $\alpha \in (0,1)$. For an arbitrary (and fixed)
 $\delta >0$ define
$$ I(u,n) = \Pr \bigg[
       \max_{1 \le k \le n}  \bigg\| \sum_{i=1}^{k} \bigg( \frac{X_{i}^{j}}{a_{n}} \,
       1_{ \big\{ \frac{|X_{i}^{j}|}{a_{n}} \le u \big\} } -  \E \bigg( \frac{X_{i}^{j}}{a_{n}} \,
       1_{ \big\{ \frac{|X_{i}^{j}|}{a_{n}} \le u \big\} } \bigg) \bigg)_{j=1,\ldots,d} \bigg\| > \delta
       \bigg].$$
Using stationarity, Chebyshev's inequality and the fact that
$|x^{j}| \leqslant \|(x^{1},\ldots,x^{d})\| \leqslant
\sum_{j=1}^{d}|x^{j}|$ we get the bound
 \begin{eqnarray}\label{e:alpha01}
   \nonumber I(u,\,n) & \le & \Pr \bigg[ \sum_{i=1}^{n} \bigg\| \bigg( \frac{X_{i}^{j}}{a_{n}} \,
       1_{ \big\{ \frac{|X_{i}^{j}|}{a_{n}} \le u \big\} } -  \E \bigg( \frac{X_{i}^{j}}{a_{n}} \,
       1_{ \big\{ \frac{|X_{i}^{j}|}{a_{n}} \le u \big\} } \bigg) \bigg)_{j=1,\ldots,d} \bigg\| > \delta
       \bigg]\\
    \nonumber & = & \delta^{-1} n \E \bigg[ \bigg\| \bigg( \frac{X_{1}^{j}}{a_{n}} \,
       1_{ \big\{ \frac{|X_{1}^{j}|}{a_{n}} \le u \big\} } -  \E \bigg( \frac{X_{1}^{j}}{a_{n}} \,
       1_{ \big\{ \frac{|X_{1}^{j}|}{a_{n}} \le u \big\} } \bigg) \bigg)_{j=1,\ldots,d} \bigg\|
       \bigg]\\
    \nonumber & \le &  2 \delta^{-1} n \sum_{j=1}^{d} \E \bigg( \frac{|X_{1}^{j}|}{a_{n}} \, 1_{ \big\{ \frac{|X_{1}^{j}|}{a_{n}}
            \le u \big\} } \bigg)\\
    \nonumber & = &  \frac{2n}{\delta} \sum_{j=1}^{d}
             \bigg[ \E \bigg( \frac{|X_{1}^{j}|}{a_{n}} \, 1_{ \big\{ \frac{|X_{1}^{j}|}{a_{n}}
            \le u, \frac{\|X_{1}\|}{a_{n}} >u \big\} } \bigg) +
             \E \bigg( \frac{|X_{1}^{j}|}{a_{n}} \, 1_{ \big\{ \frac{|X_{1}^{j}|}{a_{n}}
            \le u, \frac{\|X_{1}\|}{a_{n}} \leqslant u \big\} } \bigg) \bigg]\\
    \nonumber & \le & \frac{2n}{\delta} \sum_{j=1}^{d} \bigg[ u \Pr
    \bigg( \frac{\|X_{1}\|}{a_{n}} >u \bigg) + \E \bigg( \frac{\|X_{1}\|}{a_{n}} \, 1_{ \frac{\|X_{1}\|}{a_{n}} \leqslant u \big\} } \bigg) \bigg]\\
    \nonumber & = & \frac{2du}{\delta} \cdot n \Pr (\|X_{1}\|>a_{n}) \cdot
      \frac{\Pr(\|X_{1}\|>ua_{n})}{\Pr(\|X_{1}\|>a_{n})}
            \cdot \bigg[ 1 + \frac{\E(\|X_{1}\| \, 1_{ \{ \|X_{1}\| \le u a_{n} \}
            })}{ua_{n}\Pr(\|X_{1}\|>ua_{n})} \bigg].\\
    & &
 \end{eqnarray}
Since $X_{1}$ is a regularly varying random variable with index
$\alpha$, it follows immediately that
$$ \frac{\Pr(\|X_{1}\|>ua_{n})}{\Pr(\|X_{1}\|>a_{n})} \to
    u^{-\alpha},$$
 as $n \to \infty$. By Karamata's theorem
 $$ \lim_{n \to \infty} \frac{\E(\|X_{1}\| \, 1_{ \{ \|X_{1}\| \le u a_{n} \}
            })}{ua_{n}\Pr(\|X_{1}\|>ua_{n})} =
            \frac{\alpha}{1-\alpha}.$$
 Thus from (\ref{e:alpha01}), taking into account
 relation (\ref{e:niz}), we get
 $$ \limsup_{n \to \infty} I(u,\,n) \le 2d \delta^{-1} u^{1-\alpha}
 \Big(1+\frac{\alpha}{1-\alpha}\Big).$$
 Letting $u \to 0$, since $1-\alpha >0$, we finally obtain
 $$ \lim_{u \downarrow 0} \limsup_{n \to \infty} I(u,\,n)=0,$$
 and relation (\ref{e:slutskycond}) holds.
 Therefore $V_{n} \dto V$ as $n \to \infty$ in
 $D([0,1], \mathbb{R}^{d})$ endowed with the weak $M_{1}$ topology.
 \qed
\end{proof}

\section{Two examples}
\label{S:examples}

\begin{example} (A $q$-dependent process).
Consider an i.i.d. sequence $(Z_t)_{t \in \mathbb{Z}}$ of regularly varying random variables with index $\alpha \in (0,2)$, and construct a lagged process
$$
X_t^{(q)} = (Z_t,\ldots, Z_{t-q}),\ {t \in \mathbb{Z}}\,,
$$
for some fixed $q \in \mathbb{N}$. Take a sequence of positive real numbers $(a_{n})$ such that
$$ n \Pr (|Z_{1}|>a_{n}) \to 1 \qquad \textrm{as} \ n \to \infty.$$ By an application of Proposition 5.1 in Basrak et al.~\cite{BDM02b} it can be seen that the random process $(X_{t})$ is jointly regularly varying. Since the sequence $(X_{t}^{(q)})$ is $q$--dependent, it is also strongly mixing, and therefore
Condition~\ref{c:mixcond} holds for any  positive integer sequence
$(r_{n})_{n \in \mathbb{N}}$ such that $r_{n} \to \infty $ and
$r_{n} / n \to 0$ as $n \to \infty$.
By the same property using Remark~\ref{R:rhomix}
one can easily see that Conditions~\ref{c:anticluster} and ~\ref{c:step6cond} hold.

It is not much more difficult to
 check the condition on the tail process $(Y_{i})_{i \in \mathbb{Z}}$ given in the statement  of Theorem~\ref{t:2}. Fix $j \in \{1,\ldots,q+1\}$, $k,l \in \mathbb{Z},\,k \neq l$, and arbitrary $r>0$. From relation (\ref{e:tailprocess}), using a standard regular variation argument, we obtain
\begin{eqnarray*}
  \Pr(Y_{k}^{j}>r, Y_{l}^{j}<-r) &=& \lim_{n \to \infty} \Pr \bigg( \frac{X_{k}^{(q)\,j}}{a_{n}}>r, \frac{X_{l}^{(q)\,j}}{a_{n}}<-r \, \bigg| \, \|X_{0}\|>a_{n}\bigg) \\[0.3em]
   &=& \lim_{n \to \infty} \Pr \bigg( \frac{Z_{k+1-j}}{a_{n}}>r, \frac{Z_{l+1-j}}{a_{n}}<-r \, \bigg| \, \|X_{0}\|>a_{n}\bigg)\\[0.3em]
   & \leqslant & \liminf_{n \to \infty} \frac{\Pr(Z_{k+1-j}>ra_{n}, Z_{l+1-j}<-ra_{n})}{\Pr(\|X_{0}\|>a_{n})}\\[0.3em]
   &=& \liminf_{n \to \infty} \frac{n \Pr(Z_{k+1-j}>ra_{n}) \Pr( Z_{l+1-j}<-ra_{n})}{n \Pr(\|X_{0}\|>a_{n})}\\[0.3em]
   &=& 0.
\end{eqnarray*}
Since $r>0$ was arbitrary, it holds that $\Pr(Y_{k}^{j}>0, Y_{l}^{j}<0) =0$, i.e. $(Y_{i}^{j})_{i \in \mathbb{Z}}$ almost surely has no two values of the opposite sign.

Thus, $(X_{t}^{(q)})$ satisfies all the conditions of Theorem~\ref{t:2}, and the corresponding partial sum processes $V_{n}(\,\cdot\,)$ converge in distribution to an $\alpha$--stable L\'{e}vy process $V(\,\cdot\,)$ under the weak $M_{1}$ topology.

Since the sequence $(X_{t}^{(q)\,j})_{t \in \mathbb{Z}}=(Z_{t+1-j})_{t\in \mathbb{Z}}$ consists of i.i.d. random variables, by the univariate functional limit theorem (see Basrak et al.~\cite{BKS}), for every $j=1,\ldots, q+1$, the univariate partial sum processes
$$ V_{n}^{j}(t) = \sum_{k=1}^{\floor{nt}} \frac{Z_{k+1-j}}{a_{n}} -
  \lfloor nt \rfloor \E  \bigg( \frac{Z_{1}}{a_{n}} 1_{ \big\{ \frac{|Z_{1}|}{a_{n}} \le 1 \big\} } \bigg), \quad t \in [0,1],$$
  converge in distribution as $n \to \infty$, in $D[0,1]$ under the $M_{1}$ topology, to an $\alpha$--stable L\'{e}vy process with characteristic triple $(0,\mu,0)$ where the measure $\mu$ is the vague limit of $n \Pr(Z_{1}/a_{n} \in \cdot)$ as $n \to \infty$. The last convergence holds also in the $J_{1}$ sense. 

Next we show that $V_{n}(\,\cdot\,)$ does not converge in distribution under the standard (or strong) $M_{1}$ topology on $D([0,1], \mathbb{R}^{q+1})$. For simplicity take $q=1$. Then $X_{t}^{(1)}=(Z_{t}, Z_{t-1})$ and $V_{n}(t)=(V_{n}^{1}(t), V_{n}^{2}(t))$ where
$$ V_{n}^{1}(t) = \sum_{k=1}^{\floor{nt}} \frac{Z_{k}}{a_{n}} -
  \lfloor nt \rfloor \E  \bigg( \frac{Z_{1}}{a_{n}} 1_{ \big\{ \frac{|Z_{1}|}{a_{n}} \le 1 \big\} } \bigg) $$
and
$$ V_{n}^{2}(t) = \sum_{k=1}^{\floor{nt}} \frac{Z_{k-1}}{a_{n}} -
  \lfloor nt \rfloor \E  \bigg( \frac{Z_{1}}{a_{n}} 1_{ \big\{ \frac{|Z_{1}|}{a_{n}} \le 1 \big\} } \bigg).$$
Observe
$$ (V_{n}^{1} - V_{n}^{2})(\,\cdot\,)= \sum_{k=1}^{\floor{n \, \cdot}} \frac{Z_{k}-Z_{k-1}}{a_{n}} = \frac{Z_{\floor{n \,\cdot}}-Z_{0}}{a_{n}} \fidi 0,$$
as $n \to \infty$, but this convergence can not be replaced by the convergence in the $M_{1}$ topology on $D([0,1], \mathbb{R})$, since as it is known, $\sup_{t\in[0,1]} Z_{\floor{nt}}/a_{n}$ converges in distribution to a nonzero limit, and $\sup_{t\in[0,1]}$ is a continuous functional in the $M_{1}$ topology. Therefore $V_{n}^{1}-V_{n}^{2}$ does not converge in distribution in $D([0,1], \mathbb{R})$ endowed with the $M_{1}$ topology.

However, if $V_{n}$ would converge in distribution to some $V$ in the standard $M_{1}$ topology, then using the fact that linear combinations of the coordinates are continuous in the same topology (see Theorem 12.7.1 in Whitt~\cite{Whitt02}) and the continuous mapping theorem, we would obtain that $V_{n}^{1}-V_{n}^{2}$ converges to $V^{1}-V^{2}$ in $D([0,1], \mathbb{R})$ endowed with the $M_{1}$ topology, which is impossible.

Our example shows that the standard $M_{1}$ convergence in the multivariate functional limit theorem excludes some very basic models. The difference with the weak $M_{1}$ convergence in this example
can be explained by different behavior of linear functions of the coordinates in the two topologies: they are continuous in the standard $M_{1}$ topology, but not in the weak $M_{1}$ topology.
One can also show that Lemma~\ref{L:contsf} does not hold if the weak $M_{1}$ topology is replaced by the standard $M_{1}$ topology.
\end{example}

\begin{example} (Stochastic recurrence equation).
Another standard class of processes satisfying our main theorem is a class of
multivariate
 stationary solutions to stochastic recurrence equations.
Here, we suppose that a $d$--dimensional random process $(X_{t})$ satisfies a stochastic recurrence equation
 \begin{equation}\label{e:SRE}
 X_{t}=A_{t}X_{t-1} + B_{t}, \qquad t \in \mathbb{Z},
 \end{equation}
 for some i.i.d. sequence $((A_{t},B_{t}))$ of random $d \times d$ matrices $A_{t}$ and $d$--dimensional vectors $B_{t}$. One can view $(X_{t})$ as a multivariate random coefficient AR(1) processes.
 For simplicity, we assume that components of $(X_t),(A_t)$ and $(B_t)$ are all nonnegative.
 For instance, the process of conditional factor variances of a factor GARCH model considered in Hafner and Preminger~\cite{HaPr09} satisfies (\ref{e:SRE}) (cf. also Basrak and Segers~\cite{BaSe}).
It is known by the work of Kesten~\cite{Ke73}, see also
  Basrak et al.~\cite{BDM02b}, Theorem 2.4, that under relatively general conditions there exists a stationary causal solution $(X_{t})$ to the stochastic recurrence equation (\ref{e:SRE}) which satisfies the multivariate regular variation condition.
 It is known further (cf.~\cite{BDM02b}) that such a process $(X_{t})$ is jointly regularly varying and satisfies Condition~\ref{c:anticluster}. If we assume that the process $(X_{t})$ is $\mu$--irreducible, then according to  Theorem 16.1.5 in Meyn and Tweedie~\cite{MeTw93}, it is also strongly mixing with geometric rate
 (cf. Theorem 2.8 in Basrak et al.~\cite{BDM02b}). Consider now time series of this form  for which the index of regular variation $\alpha \in (0,2)$.
   Since the components of $X_{t}$ are assumed to be nonnegative, it trivially holds that the tail process $(Y_{i})$ of $(X_{t})$ satisfies the condition that $(Y_{i}^{j})_{i}$ has no two values of the opposite sign for every $j=1,\ldots,d$. Therefore, if additionally Condition~\ref{c:step6cond} holds when $\alpha \in [1,2)$, then by Theorem~\ref{t:2} the partial sum stochastic process
  $$ V_{n}(t) =
  \sum_{k=1}^{\floor{nt}} \frac{X_{k}}{a_{n}} -
  \lfloor nt \rfloor \E \bigg( \bigg( \frac{X_{1}^{j}}{a_{n}} 1_{ \big\{ \frac{|X_{1}^{j}|}{a_{n}} \le 1 \big\} } \bigg)_{j=1,\ldots,d} \bigg), \quad t \in [0,1], $$
converges in $D([0,1], \mathbb{R}^{d}) $ with the weak $M_{1}$ topology to an
$\alpha$--stable L\'{e}vy process $V(\,\cdot\,)$.
\end{example}


%
%

\section*{Acknowledgements}
Bojan Basrak's research was partially supported by the research grant MZOS nr.
037-0372790-2800 of the Croatian government.


\begin{thebibliography}{00}

\bibitem{AvTa92}
Avram, F. and Taqqu, M. (1992) Weak convergence of sums of moving
averages in the $\alpha$--stable domain of attraction. \emph{Ann.
Probab.} {\bf 20}, 483--503.

\bibitem{BaJaMiWi09}
Bartkiewicz, K., Jakubowski, A., Mikosch, T. and Wintenberger, O. (2011) Stable limits for sums of dependent infinite variance random variables. \emph{Probab. Theory Relat. Fields} {\bf 150}, 337--372.

\bibitem{BDM02a}
Basrak, B., Davis, R.A. and Mikosch, T. (2002) A characterization of multivariate regular variation. \emph{Ann. Appl. Probab.} {\bf12}, 908--920.

\bibitem{BDM02b}
Basrak, B., Davis, R.A. and Mikosch, T (2002) Regular variation
of GARCH processes. \emph{Stoch. Process. Appl.} {\bf 99},
95--115.

\bibitem{BKS}
Basrak, B., Krizmani\'{c}, D. and Segers, J. (2012) A functional limit theorem for partial sums
of dependent random variables with infinite variance. \emph{Ann. Probab.} {\bf 40}, 2008--2033.

\bibitem{BaSe}
Basrak, B. and Segers, J. (2009) Regularly varying multivariate time
series. \emph{Stoch. Process. Appl.} {\bf 119}, 1055--1080.

\bibitem{Bertoin96}
Bertoin, J. (1996) \emph{L\'{e}vy Processes}. Cambridge Tracts in
Mathematics Vol. 121, Cambridge University Press, Cambridge.

\bibitem{Bradley05}
Bradley, R.C. (2005) Basic Properties of Strong Mixing Conditions.
A Survey and Some Open Questions. \emph{Probab. Surv.}, Vol.2,
107--144.

\bibitem{Da83}
Davis, R.A. (1983) Stable limits for partial sums of dependent
random variables. \emph{Ann. Probab.} {\bf 11}, 262--269.

\bibitem{DaHs95}
Davis, R.A. and Hsing, T. (1995) Point process and partial sum
convergence for weakly dependent random variables with infinite
variance. \emph{Ann. Probab.} {\bf 23}, 879--917.

\bibitem{DaMi98}
Davis, R.A. and Mikosch, T. (1998) The sample autocorrelations of
heavy-tailed processes with applications to ARCH. \emph{Ann.
Statist.} {\bf 26}, 2049--2080.

\bibitem{DeJa89}
Denker, M. and Jakubowski, A. (1989) Stable limit distributions for
strongly mixing sequences. \emph{Stat. Probab. Lett.} {\bf 8},
477--483.

\bibitem{Durrett78}
Durrett, R. and Resnick, S.I. (1978) Functional limit theorems for
dependent variables. \emph{Ann. Probab.} {\bf 6}, 829--846.

\bibitem{Em97}
Embrechts, P., Kl\"uppel\-berg, C. and Mikosch, T. (1997)
\emph{Modelling Extremal Events}. Springer-Verlag, Berlin.

\bibitem{HaPr09}
Hafner, C.M. and Preminger, A. (2009) Asymptotic theory for a factor GARCH model. \emph{Economet. Theor.} {\bf 25}, 336--363.

\bibitem{HL06}
Hult, H. and Lindskog, F. (2006) On Kesten's counterexample to the Cram\'{e}r-Wold
device for regular variation. \emph{Bernoulli} {\bf 12}, 133--142.

\bibitem{JaKo89}
Jakubowski, A. and Kobus, M. (1989) $\alpha$--stable limit theorems for sums of dependent
random vectors. \emph{J. Multivariate Anal.} {\bf 29}, 219-251.

\bibitem{Kallenberg97}
Kallenberg, O. (1997) \emph{Foundations of Modern Probability}.
Springer-Verlag, New York.

\bibitem{Ke73}
Kesten, H. (1973) Random difference equations and renewal theory for products of random matrices. \emph{Acta Math.} {\bf 131}, 207--248.

\bibitem{Kr10}
Krizmani\'c, D. (2010) Functional limit theorems for weakly
dependent regularly varying time series. Ph.D. thesis.
http://www.math.uniri.hr/$\sim$dkrizmanic/DKthesis.pdf. Accessed 14 June 2012.

\bibitem{MeTw93}
Meyn, S.P. and Tweedie, R.L. (1993) \emph{Markov Chains and Stochastic Stability}. Springer, London.

\bibitem{Peligrad99}
Peligrad, M. (1999) Convergence of stopped sums of weakly dependent
random variables. \emph{Electron. J. Probab.} {\bf 4}, 1--13.

\bibitem{Ph80}
Phillip, W. (1980) Weak and $L^p$-invariance principles for sums of $B$-valued random variables. \emph{Ann. Probab.} {\bf 8}, 68--82.

\bibitem{Ph86}
Phillip, W. (1986) Correction to "Weak and $L^p$-invariance principles for sums of $B$-valued random variables". \emph{Ann. Probab.} {\bf 14}, 1095--1101.

\bibitem{Resnick86}
Resnick, S.I. (1986) Point processes, regular variation and weak
convergence. \emph{Adv. Appl. Probab.} {\bf 18}, 66--138.

\bibitem{Resnick07}
Resnick, S.I. (2007) \emph{Heavy-Tail Phenomena: Probabilistic nad
Statistical Modeling}. Springer Science+Business Media LLC, New
York.

\bibitem{Rv62}
Rva\v{c}eva, E.L. (1962) On domains of attraction of multi-dimensional distributions. In: \emph{Select. Transl. Math. Statist. and Probability}, Vol. 2, pp. 183--205. Am. Math. Soc., Providence.

\bibitem{Sato99}
Sato, K. (1999) \emph{L\'{e}vy Processes and Infinitely Divisible
Distributions}. Cambridge Studies in Advanced Mathematics, Vol. 68.
Cambridge University Press, Cambridge.

\bibitem{Sk57}
Skorohod, A.V. (1957) Limit theorems for stochastic processes with independent
increments. \emph{Theory Probab. Appl.} {\bf 2}, 138--171.

\bibitem{Ty10}
Tyran-Kami\'{n}ska, M. (2010) Convergence to L\'{e}y stable processes under some weak
dependence conditions. \emph{Stoch. Process. Appl.} {\bf 120}, 1629--1650.

\bibitem{Whitt02}
Whitt, W. (2002) \emph{Stochastic-Process Limits}. Springer-Verlag
LLC, New York.


\end{thebibliography}
\end{document}